\documentclass[a4paper,11pt]{article}
\usepackage[utf8]{inputenc}
\usepackage[all]{xy}
\usepackage{amssymb}
\usepackage{algorithmicx}
\usepackage[english]{babel}
\usepackage{graphics}
\usepackage{graphicx}
\graphicspath{ {images/} }
\usepackage{epstopdf}
\usepackage{amsmath}
\usepackage{amsfonts}
\usepackage{mathbbol}
\usepackage{amsthm} 
\usepackage{mathrsfs}
\usepackage{tikz-cd}
\usepackage{authblk}
\usetikzlibrary{shapes, shadows, arrows}
\usepackage[noadjust]{cite}

\newtheorem{theorem}{Theorem}[section]

\newtheorem{definition}[theorem]{Definition}
\newtheorem{proposition}[theorem]{Proposition}
\newtheorem{corollary}[theorem]{Corollary}
\newtheorem{lemma}[theorem]{Lemma}
\newtheorem{remark}[theorem]{Remark}

\newtheorem*{theorem*}{Theorem}
\newtheorem*{thma}{Theorem A}
\newtheorem*{thmb}{Theorem B}

\DeclareMathOperator{\im}{im}

\DeclareMathOperator{\id}{id}

\DeclareMathOperator{\cost}{cost}

\DeclareMathOperator{\coim}{coim}

\title{A Closed Formula for the Interleaving Distance of Rectangle Persistence Modules}
\author[1]{Mehmet Ali Batan}
\author[2]{Claudia Landi}
\author[3]{Mehmetcik Pamuk}

\affil[1]{Department of Mathematics, Middle East Technical University, Ankara, Turkey, \authorcr \texttt{alibatan@metu.edu.tr}}

\affil[2]{Department of Sciences and Methods for Engineering, University of Modena and Reggio Emilia, Italy, \authorcr \texttt{claudia.landi@unimore.it}}

\affil[3]{Department of Mathematics, Middle East Technical University, Ankara, Turkey, \authorcr \texttt{mpamuk@metu.edu.tr}}

\date{\today}

\begin{document}

\maketitle

\begin{abstract}
We give formulas for calculating the interleaving distance between rectangle persistence modules that depend solely on the geometry of the underlying rectangles. Moreover, we extend our results to calculate the bottleneck distance for rectangle decomposable persistence modules.
\end{abstract}

\section{Introduction}

Topological data analysis is a recently emerging and fast-growing field for analyzing complex data using geometry and topology. Persistent homology is a powerful tool in topological data analysis for investigating data structure. Persistent homology studies topological features of a space that persist for some range of parameter values.  It has found a wide range of applications to biology \cite{topaz-ziegelmeier-halverson}, information networks \cite{aktas-akbas-fatmaoui}, finance \cite{ismail-noorani-ismail-razak-alias} etc., as well as to other parts of mathematics.

Let $\mathbb{k}$ be a fixed field and $P$ be a poset.  A ($P$-indexed) persistence module $\mathcal{M}$ is a functor $\mathcal{M}\colon P \to \textrm{\textbf{Vec}}$ where \textbf{Vec} denotes the category of $\mathbb{k}$-vector spaces.  Throughout the paper, we consider the poset of real numbers $\mathbb{R}$ with its standard total order and the product poset $\mathbb{R}^n$.  Modules of the form $\mathbb{R}^n \to \textrm{\textbf{Vec}}$ are known as $n$-parameter persistence modules.

In the $1$-parameter setting ($n=1$), persistence modules decompose in an essentially unique way into simple summands called interval modules. The decomposition is specified by a discrete invariant called a barcode.  In contrast, the algebraic structure of a multiparameter ($n\geq2$) persistence module can be far more complex.  While decomposition theorems exist, the underlying quivers are of wild-representation type, meaning that their sets of indecomposables are hard to classify.  Because of this reason, it is natural to look for a class of persistence modules where we can hope for results similar to the $1$-parameter case. 
Another useful property of $1$-parameter persistence modules is that they are completely described up to isomorphism by the rank invariant \cite{carlsson-zomorodian}. Botnan et al. \cite{botnan-lebovici-oudot}  prove that the rank invariant is complete on the class of rectangle-decomposable $2$-parameter persistence modules and provide algorithms to determine efficiently whether a given $2$-parameter persistence module belongs to this class.

Distances are tools for quantifying the (dis)similarity between persistence modules.  They play an essential role in both theory and application.  Different types of distances have been proposed on various types of persistence modules with values in 
\textbf{Vec}.  Persistence modules are usually compared using the interleaving distance due to their universality property over prime fields \cite{Lesnick}.   The interleaving distance for $1$-parameter persistence modules was introduced by Chazal et al. \cite{chazal et al}.  On the other hand, the bottleneck distance is the standard metric on barcodes; that is, it depends on the barcode decomposition of the given modules.
For the $1$-parameter persistence modules, it is now well-known that these two distances are equal to each other \cite{Lesnick}. Hence, computation of the interleaving distance can be done by computing the bottleneck distance with efficient algorithms.
Since a canonical definition of a barcode is unavailable for multiparameter persistence modules, the definition of bottleneck distance does not admit an immediate extension to the multidimensional setting. Nevertheless, it is still possible to define meaningful generalizations of the bottleneck distance between multiparameter persistence modules whose decompositions are given (see, for example, \cite[Definition~2.9]{Bjerkevik} for the interval decomposable modules).  Unfortunately, the interleaving distance and the bottleneck distance are no longer equal.  The bottleneck distance is an upper bound for the interleaving distance (see \cite{Bjerkevik} for details).

It is known that the interleaving distance is NP-hard to compute for persistence modules valued in \textbf{Vec} \cite{bjerkevik-botnan}. So, it is natural to consider a subclass of persistence modules to obtain a formula for the interleaving distance. For a polynomial time algorithm that computes the bottleneck distance for $2$-parameter interval decomposable modules, see \cite{Dey}. Because of the results mentioned above, rectangle decomposable persistence modules are another interesting class.

In this paper, we give a closed formula to calculate the interleaving distance between rectangle persistence modules that depend solely on the geometry of the underlying rectangles.
Let  $\mathcal{M}$ and $\mathcal{N}$ be two rectangle persistence modules with underlying rectangles $R_\mathcal{M}=(a_1, b_1) \times (a_2, b_2)$ and $R_\mathcal{N}=(c_1, d_1) \times (c_2, d_2)$ respectively. The first main result of the paper is as follows.

\begin{thma}
Let $\mathcal{M}$ and $\mathcal{N}$ be two rectangle persistence modules.  We have
\small \[
d_I(\mathcal{M}, \mathcal{N}) = \min \Bigg\{ \max \Big\{ \min_{i=1, 2} \frac{b_i-a_i}{2}, \min_{i=1, 2} \frac{d_i-c_i}{2} \Big\}, \max \Big\{ \| c-a \|_{\infty}, \| d-b \|_{\infty} \Big\}    \Bigg\}.
\]      
\end{thma}

Theorem A can be easily extended to the case when $n > 2$ (see Corollary~\ref{interleavingfor$n$-parameter}).  We also want to note that by   Proposition~\ref{closureequality} and Corollary~\ref{interiorequality}, our result holds regardless of whether the underlying rectangles are open, closed, or neither.

Furthermore, thanks to the fact that if $\mathcal{M}$ and $\mathcal{N}$ are two interval persistence modules, then the interleaving distance is equal to the bottleneck distance (see Corollary~\ref{indecomposableequality}), the formula above can be used to calculate the bottleneck distance between two rectangle persistence modules in essentially the same way as for the bottleneck distance in the $1$-parameter case (see \cite[p. 50]{Oudot}). Since, for rectangle decomposable persistence modules, the bottleneck distance differs from the interleaving distance, we outline below an alternative formula for the bottleneck distance in that case.

Let $\mathcal{M}$ and $\mathcal{N}$ be two rectangle decomposable persistence modules with decompositions 
$$
\displaystyle \mathcal{M}=\bigoplus_{i=1}^{m}\mathcal{M}_i ~\textrm{and}~ \displaystyle \mathcal{N}=\bigoplus_{j=1}^{k}\mathcal{N}_j.$$  
Let us denote the underlying rectangle of the persistence module $\mathcal{M}_i$ by $R_i=(a_1^i, b_1^i)\times(a_2^i, b_2^i)$
and $\mathcal{N}_j$ by $Q_j=(c_1^j, d_1^j)\times(c_2^j, d_2^j)$
where $a^i=(a_1^i, a_2^i)$, $b^i=(b_1^i, b_2^i)$, $c^j=(c_1^j, c_2^j)$ and $d^j=(d_1^j, d_2^j)$. The second main result is as follows.

\begin{thmb}
Let $\mathcal{M}$ and $\mathcal{N}$ be two rectangle decomposable persistence modules given as above. Let $S$ be the set of all partial multibijections between the barcodes $B(\mathcal{M})$ and $B(\mathcal{N})$. Let $\sigma \in S$ and let $\mathbf{I}'(\sigma)=B(\mathcal{M})- \coim \sigma$, $\mathbf{J}'(\sigma)=B(\mathcal{N})- \im \sigma$. Then, the bottleneck distance is equal to {\footnotesize
\[
 \displaystyle \min_{\sigma \in S} \max \left\{  \max_{R_i \in \coim \sigma} \Big\{  d_I\Big(\mathcal{M}_i, \sigma(\mathcal{M}_i)\Big) \Big\}, \max_{R_{i'} \in \mathbf{I}'(\sigma)} \Big\{ \min_{s=1,2} \frac{b_s^{i'}-a_s^{i'}}{2} \Big\}, \max_{Q_{j'} \in \mathbf{J}'(\sigma)} \Big\{ \min_{s=1,2} \frac{d_s^{j'}-c_s^{j'}}{2} \Big\} \right\}
\]}
where the interleaving distance between $\mathcal{M}_i$ and $\sigma(\mathcal{M}_i)$ is equal to {\footnotesize
\[
\min \Bigg\{ \max \Big\{ \min_{s=1,2} \frac{b_s^i-a_s^i}{2}, \min_{s=1,2} \frac{d_s^{\sigma(i)}-c_s^{\sigma(i)}}{2} \Big\}, \max \Big\{ \| c^{\sigma(i)}-a^i \|_{\infty}, \| d^{\sigma(i)}-b^i \|_{\infty} \Big\}  \Bigg\}.
\]}
\end{thmb}

\vskip 0.2cm
\noindent{\bf Acknowledgements.} 
 This research was supported by the T\"{U}B\.{I}TAK Scholarships 2214-A - International Research Fellowship Programme for PhD Students and 220N359.  The results in this paper are part of the first author’s PhD thesis \cite{Batan}. The first author would like to thank Vanessa Robins and the Applied Topology and Geometry Group members at Australian National University for their hospitality during the last year of his PhD studies and their valuable comments. 

\section{Preliminaries}

\subsection{Notations and Conventions}\label{NotionsandConventions}
For $n \geq 1$, let us start with defining a partial order on $ \mathbb{R}^{n}$ by taking 
$$u=(u_1, \dots, u_n) \preceq  v = (v_1, \ldots, v_n)$$ if and only if $ u_i \leq v_i $ for all $i=1, 2, \ldots, n$, and $u \prec  v$ if and only $ u_i < v_i $ for all $i=1, 2, \ldots, n$.  Note that $u \succ v$ is not the negation of $u \preceq v$.
Also, let us endow $\mathbb{R}^{n}$ with the max-norm, that is  
$$
\displaystyle \| u \|_{\infty} \doteq  \max_{1 \leq i \leq n}  \left\{ |u_i| \right\} 
$$ 
\noindent
for all $u\in \mathbb{R}^{n}$. The metric induced by the max-norm is 
$$
d_\infty(u,v) \doteq \max \left\{|u_1-v_1|, \ldots, |u_n-v_n| \right\}
$$ 
\noindent
for all $u, v \in \mathbb{R}^{n}$.
Let us define the \textbf{extended real line}, and more generally, the \textbf{extended space}, as $\overline{\mathbb{R}}\doteq \mathbb{R}\cup \{ -\infty, +\infty \}$ and $\displaystyle \overline{\mathbb{R}}^n\doteq \prod_{i=1}^n \overline{\mathbb{R}}$, respectively.

Throughout the paper, for any $a\in \mathbb{R}$, we follow the convention that

\begin{equation}\label{conventionslist1}
\begin{split}
 a+(\pm \infty)=(\pm \infty)+a=\pm \infty, \\
   -\infty < a < +\infty,
\end{split}
\end{equation}
and, for any $b \in \overline{\mathbb{R}}$,
\begin{equation}\label{conventionslist2}
\begin{split}
& -\infty \leq b \leq +\infty \,.
\end{split}
\end{equation}
Moreover, we follow the convention that
\begin{equation}\label{conventionslist3}
\begin{split}
& (\pm \infty)-(\pm \infty)=0, \\
& (\pm \infty)+(\pm \infty)=\pm \infty, \\
& |\pm \infty|=+\infty. 
\end{split}
\end{equation}

\subsection{Persistence Modules}

\begin{definition}\label{persistencemodule}
An $n$-parameter \textbf{persistence module} $\mathcal{M}$ over a field $\mathbb{k}$ is a family of vector spaces $\displaystyle \{\mathcal{M}_u \}_{u \in \mathbb{R}^{n}}$ over $\mathbb{k}$ together with a family of linear maps called \textbf{transition maps}.
$$
\{ \varphi_\mathcal{M}(u, v)\colon \mathcal{M}_u \rightarrow \mathcal{M}_v, \ u\preceq v \in \mathbb{R}^{n}\}
$$ 
such that  for every $u\preceq v\preceq w \in \mathbb{R}^{n}$, we have 
\begin{itemize}
    \item [(i)]  $\varphi_\mathcal{M}(v, w)\circ \varphi_\mathcal{M}(u, v)=\varphi_\mathcal{M}(u, w)$,
    \item [(ii)] $\varphi_\mathcal{M}(u, u)=\id_{\mathcal{M}_u}$.
\end{itemize}
\end{definition}

We say that the $n$-parameter persistence module  $\mathcal{M}$ is the \textbf{zero persistence module} if $\mathcal{M}_u$ is the zero vector space for all $u\in \mathbb{R}^n$.

\begin{definition} \label{morphismbetweenpersistencemodules}
A \textbf{morphism} $f \colon \mathcal{M} \rightarrow \mathcal{N}$ between two $n$-parameter persistence modules $ \mathcal{M}$ and $\mathcal{N}$ is a collection of linear maps $\{f_u \colon \mathcal{M}_u \rightarrow \mathcal{N}_u \}$ such that the following diagram commutes for all $u \preceq v \in \mathbb{R}^n$.

\[ 
\begin{tikzcd}
\mathcal{M}_u \arrow{r}{\varphi_\mathcal{M}(u, v)} \arrow[swap]{d}{f_u} & \mathcal{M}_v \arrow{d}{f_v} \\
\mathcal{N}_u \arrow{r}{\varphi_\mathcal{N}(u, v)}& \mathcal{N}_v
\end{tikzcd}
\]
\end{definition}

\begin{definition} \label{isomorphismofpersistencemodules}
We say that the persistence modules $ \mathcal{M}$ and $\mathcal{N}$ are isomorphic, denoted by $\mathcal{M} \simeq \mathcal{N}$, if there exist two morphisms $f \colon \mathcal{M} \rightarrow \mathcal{N}$ and $g \colon \mathcal{N}  \rightarrow \mathcal{M}$ such that $f\circ g$ and $g\circ f$ are identity maps.
\end{definition}

\begin{definition} \label{shift}
Let $\mathcal{M}$ be an $n$-parameter persistence module. Let $\Vec{\epsilon}=(\epsilon, \epsilon, \ldots, \epsilon) \in \mathbb{R}^n$. For any $\Vec{\epsilon}\in \mathbb{R}^n$ with $\epsilon\geq 0$, the \textbf{$\epsilon$-shift} of the persistence module $\mathcal{M}$ is the $n$-parameter persistence module $\mathcal{M}(\Vec{\epsilon})$ defined as

\begin{itemize}
    \item $\mathcal{M}(\Vec{\epsilon})_u\doteq \mathcal{M}_{u+\Vec{\epsilon}} \ $ for any $u\in \mathbb{R}^n$, and
  
    \item $\varphi_{\mathcal{M}(\Vec{\epsilon})}(u, v)\doteq \varphi_\mathcal{M}(u+\Vec{\epsilon}, v+\Vec{\epsilon}) \ $ for any $u\preceq v \in \mathbb{R}^{n}$.
\end{itemize}

\end{definition}

\begin{definition}
Let $\mathcal{M}$ be a persistence module, we call $\displaystyle \mathcal{M}=\mathcal{M}_1 \oplus \mathcal{M}_2$ a \textbf{direct sum} of the persistence modules $\mathcal{M}_1$ and $\mathcal{M}_2$ where 
\begin{itemize}
    \item $\mathcal{M}_u \simeq ({\mathcal{M}_1})_u \oplus ({\mathcal{M}_2})_u \ $ for any  $u\in \mathbb{R}^{n}$, and
    \item $\varphi_{\mathcal{M}}(u, v)\doteq \left( {\begin{array}{cc}
   \varphi_{\mathcal{M}_1}(u, v) & 0 \\
   0 & \varphi_{\mathcal{M}_2}(u, v) 
  \end{array}} \right)$ for any $u\preceq v \in \mathbb{R}^{n}$.
\end{itemize}
Analogously, we call $\displaystyle \mathcal{M}=\bigoplus \mathcal{M}_i$ a \textbf{direct sum} of the persistence modules $\mathcal{M}_i$ for some collection of persistence modules $\{\mathcal{M}_i\}$ satisfying the same structure.

We say that $\mathcal{M}$ is an \textbf{indecomposable persistence module} if $\mathcal{M}$ is a non-zero persistence module and cannot be written as a direct sum of two non-zero persistence modules. By the Krull–Schmidt theorem \cite{Atiyah}, there exists an essentially unique (up to permutation and isomorphism) decomposition $\displaystyle \mathcal{M}=\bigoplus \mathcal{M}_i$ with every $\mathcal{M}_i$ being indecomposable.
\end{definition}

\begin{definition} [\cite{Bjerkevik}] \label{interval}
We say that $I \subseteq \mathbb{R}^{n} $ is an $n$-parameter connected \textbf{interval} if 
	 \begin{itemize}
	 \item $I$ is non-empty,
	 \item If $u ,v \in I$ and $u \preceq w \preceq v$, then $w \in I$,
	 \item  If $u , v \in I$, then  for some $m\in \mathbb{N}$ there exist $u_1, u_2, \ldots, u_{2m} \in I$ such that $u \preceq u_1 \succeq u_2 \preceq \ldots \succeq  u_{2m} \preceq v  $.
	 \end{itemize}
\end{definition}

Throughout the paper, we denote the topological closure of an interval $I$ as a subset of $\mathbb{R}^{n}$ by $\Bar{I}$, and the interior of an interval $I$ as $I^\mathrm{o}$.

\begin{definition}[\cite{Bjerkevik}] \label{intervalpersistencemodule}
We say that a persistence module $\mathcal{M}$ is an \textbf{interval persistence module} if 
 \begin{itemize}
	 \item for some interval $I\subseteq \mathbb{R}^{n} $, $\mathcal{M}_u = \mathbb{k}$ for every $u \in I$, and $\mathcal{M}_u = 0$ for every $u \notin I$.
	 \item $\varphi_\mathcal{M}(u, v) = \id_\mathbb{k}$ for every $u \preceq v \in I$.
 \end{itemize}
Here, $I$ is called the underlying interval of the persistence module $\mathcal{M}$.
\end{definition}

We may denote an interval persistence module $\mathcal{I}$ with underlying interval $I$ by $\mathcal{I}^{I}$.

\begin{definition}[\cite{Bjerkevik}]
If $ \displaystyle \mathcal{M} \simeq \bigoplus_{I \in B} \mathcal{I}^{I}$ for a multiset $B$ of intervals of $\mathbb{R}^{n}$, we say that $\mathcal{M}$ is an \textbf{interval decomposable persistence module} and the multiset $ B(\mathcal{M}) \doteq B$ is the \textbf{barcode} of $\mathcal{M}$. 
\end{definition}
 
If $ \displaystyle \mathcal{M} \simeq \bigoplus_{I \in B(\mathcal{M})} \mathcal{I}^{I}$ is an interval decomposable persistence module, then $ \overline{\mathcal{M}}$ denotes the \textbf{closed interval decomposable persistence module} $ \displaystyle \bigoplus_{I \in B(\mathcal{M})} \mathcal{I}^{\Bar{I}}$. Similarly, $\mathcal{M}^\mathrm{o}$ denotes the \textbf{open interval decomposable persistence module} $ \displaystyle \bigoplus_{I \in B(\mathcal{M})} \mathcal{I}^{{I}^\mathrm{o}}$.

\begin{remark}
Any interval persistence module $\mathcal{M}$ is indecomposable \cite[Proposition~2.2]{Botnan}. 
\end{remark}

Now, let us define rectangles, rectangle persistence modules, and rectangle decomposable persistence modules.

\begin{definition} \label{rectanglepersistencemodule}
We say that $R \subseteq \mathbb{R}^n$ is an $n$-parameter \textbf{rectangle} if $R=I_1 \times I_2 \times\ldots \times I_n$ where each $I_i$ is a $1$-parameter interval. We say that $\mathcal{M}$ is an $n$-parameter \textbf{rectangle persistence module} if it is an interval persistence module whose $n$-parameter underlying interval is an $n$-parameter rectangle $R$. We will call $R$ as the underlying rectangle of the rectangle persistence module $\mathcal{M}$. We say that a persistence module is a \textbf{rectangle decomposable persistence module} if it decomposes into a direct sum of only rectangle persistence modules.
\end{definition}

\begin{definition} 
A \textbf{free interval} generated by $u \in \mathbb{R}^{n}$ is an interval of the form $\langle u \rangle \doteq \{ v \in \mathbb{R}^{n}: u \preceq v \} $ and a \textbf{free interval persistence module} or shortly  \textbf{free persistence module} is an interval persistence module, denoted by $\mathcal{F}^{\langle u \rangle}$ with $\langle u \rangle$ as its underlying interval. A \textbf{free decomposable persistence module} $\mathcal{F}$ is a persistence module whose indecomposable summands are all free persistence modules.  Note that every free persistence module is an interval persistence module, but the converse may not be correct.
\end{definition}

\begin{definition}
[\cite{Corbet}]
A persistence module $\mathcal{M}$ is said to be a \textbf{finitely generated} persistence module if and only if there exists an epimorphism $\displaystyle \phi: \bigoplus_{i=1}^m\mathcal{F}_i\to \mathcal{M}$ where $\mathcal{F}_1, \ldots, \mathcal{F}_m$ are free persistence modules.  Furthermore, a persistence module  $\mathcal{M}$ is called a \textbf{finitely presented} persistence module if it is finitely generated and $\ker \phi$ is also finitely generated. Equivalently, $\mathcal{M}$ is a finitely presented persistence module if and only if there exists an exact sequence 
$$
\displaystyle \bigoplus_{j=1}^n\mathcal{G}_j \to \bigoplus_{i=1}^m\mathcal{F}_i \xrightarrow{\phi}\mathcal{M} \to 0
$$
where $\mathcal{G}_1, \ldots, \mathcal{G}_n$ are also free persistence modules.
\end{definition}

\subsection{Interleaving Distance}

In this subsection, we first define the interleaving distance for multiparameter persistence modules and then give some of its properties.

\begin{definition}[\cite{Lesnick}] \label{interleaving}
Let $\epsilon$ be a non-negative real number and let $\vec{\epsilon} \doteq (\epsilon, \epsilon, \ldots, \epsilon) \in \mathbb{R}^{n}$.  
An \textbf{$\epsilon $-interleaving} between $n$-parameter persistence modules  $\mathcal{M}$ and $\mathcal{N}$ is a collection of morphisms 
$f_{u}\colon \mathcal{M}_u \to \mathcal{N}_{u+\vec{\epsilon}}$ and $g_{u}\colon \mathcal{N}_u \to \mathcal{M}_{u+\vec{\epsilon}}$  such that 
the four diagrams 
\vskip .5cm		

	\begin{center}
\begin{equation} \label{squarediagramsfirst}
 	        \begin{tikzcd}                       
			    \mathcal{M}_u \arrow[r] \arrow["f_u" , d]    
				& \mathcal{M}_v \arrow[d, "f_v"] \\
				\mathcal{N}_{u+ \vec{\epsilon}} \arrow[r]
				& \mathcal{N}_{v+ \vec{\epsilon}}
			\end{tikzcd}
			\begin{tikzcd}
				\mathcal{N}_u \arrow[r] \arrow["g_u" , d]
				& \mathcal{N}_v \arrow[d, "g_v"] \\
			  \mathcal{M}_{u+ \vec{\epsilon}} \arrow[r]
				& \mathcal{M}_{v+ \vec{\epsilon}}
			\end{tikzcd}
			\vspace*{7mm}
\end{equation}
\begin{equation} \label{trianglediagramsfirst}			
			\begin{tikzcd}[row sep=small]
				& \mathcal{N}_{u+ \vec{\epsilon}} \arrow[dd, "g_{u+ \vec{\epsilon}}"] \\
				\mathcal{M}_u \arrow[ur, "f_u"] \arrow[dr] & \\
				&\mathcal{M}_{u+ 2\vec{\epsilon}}
			\end{tikzcd}
			\begin{tikzcd}[row sep=small]
				&\mathcal{M}_{u+ \vec{\epsilon}}  \arrow[dd, "f_{u+ \vec{\epsilon}}"] \\
				\mathcal{N}_u \arrow[ur, "g_u"] \arrow[dr] & \\
				&\mathcal{N}_{u+ 2\vec{\epsilon}} 
			\end{tikzcd}
\end{equation}
\end{center}
commute for all $u \preceq v \in \mathbb{R}^{n}$.	
\end{definition}
\vskip .5cm
The non-labelled maps are the transition maps mentioned in Definition~\ref{persistencemodule}. \\
 
Note that, if $\mathcal{M}$ and $\mathcal{N}$ are  $\epsilon$-interleaved, then they are $\delta$-interleaved for any $\epsilon \leq \delta $. 
Therefore, one can define the \textbf{interleaving distance} by 
$$ 
d_I(\mathcal{M}, \mathcal{N}) \doteq \inf \{\epsilon \in [0, +\infty) : \mathcal{M} \ \textrm{and} \ \mathcal{N} \ \textrm{are} \ \epsilon \textrm{-interleaved} \}.
$$ 
If no such $\epsilon$ exists, then we put $d_I(\mathcal{M}, \mathcal{N}) \doteq + \infty$.

\begin{remark}
Note that if the persistence modules $\mathcal{M}$ and $\mathcal{N}$ are $0$-interleaved, then $d_I(\mathcal{M}, \mathcal{N})=0$. However, the following example shows that the converse is not always true (see  \cite[Example~2.3]{Lesnick}). 
\end{remark}

Thanks to the following theorem, the converse statement of the remark above is also true for finitely presented persistence modules.

\begin{theorem}[\cite{Lesnick}] \label{closure}
If  $\mathcal{M}$  and  $\mathcal{N}$ are finitely presented $n$-parameter persistence modules and $d_I(\mathcal{M}, \mathcal{N}) = \epsilon$, then $\mathcal{M}$ and $\mathcal{N}$ are $\epsilon$-interleaved. 
\end{theorem}

Considering the case $\epsilon = 0$, it follows from the closure theorem that $d_I$ restricts to a metric on isomorphism classes of finitely presented multidimensional persistence modules.

\begin{definition}\label{epsilontrivial}
Let $\Vec{\epsilon}=(\epsilon, \epsilon, \ldots, \epsilon) \in \mathbb{R}^n$ with $\epsilon \geq 0$. An $n$-parameter persistence module $\mathcal{M}$ is called \textbf{$\epsilon$-significant} if $\varphi_{\mathcal{M}}(u, u + \Vec{\epsilon}) \neq 0$ for some $u \in \mathbb{R}^n$, and \textbf{$\epsilon$-trivial} otherwise.
\end{definition}

\begin{proposition}\label{epsilontrivialrectangle}
    Let $\mathcal{M}$ be an $n$-parameter rectangle persistence module with underlying rectangle $$R_\mathcal{M}=(a_1, b_1)\times (a_2, b_2)\times \ldots \times (a_n, b_n).$$ 
    Then, $\mathcal{M}$ is $\epsilon$-trivial for $\displaystyle \epsilon \geq \min_{1 \leq i \leq n} \{b_i-a_i\}$.
\end{proposition}

\begin{proposition}\label{2epsilontrivialiffepsiloninterleaved}
An $n$-parameter persistence module $\mathcal{M}$ is $2\epsilon$-trivial if and only if it is $\epsilon$-interleaved with the zero persistence module.
\end{proposition}
 
\begin{proof}
Let $\mathcal{M}$ be an $n$-parameter persistence module and let $\mathcal{N}$ be the zero persistence module, that is $\mathcal{N}_u=0$ for all $u \in \mathbb{R}^n$. Consider the following diagrams.
    
\begin{center}
\begin{equation} \label{squarediagrams2}
 	        \begin{tikzcd}                       
			    \mathcal{M}_u \arrow[r] \arrow["f_u" , d]    
				& \mathcal{M}_v \arrow[d, "f_v"] \\
				0=\mathcal{N}_{u+ \vec{\epsilon}} \arrow[r]
				& \mathcal{N}_{v+ \vec{\epsilon}}=0
			\end{tikzcd}
			\begin{tikzcd}
				0=\mathcal{N}_u \arrow[r] \arrow["g_u" , d]
				& \mathcal{N}_v=0 \arrow[d, "g_v"] \\
			  \mathcal{M}_{u+ \vec{\epsilon}} \arrow[r]
				& \mathcal{M}_{v+ \vec{\epsilon}}
			\end{tikzcd}
			\vspace*{7mm}
\end{equation}
\begin{equation} \label{trianglediagrams2}			
			\begin{tikzcd}[row sep=small]
				& \mathcal{N}_{u+ \vec{\epsilon}}=0 \arrow[dd, "g_{u+ \vec{\epsilon}}"] \\
				\mathcal{M}_u \arrow[ur, "f_u"] \arrow[dr] & \\
				&\mathcal{M}_{u+ 2\vec{\epsilon}}
			\end{tikzcd}
			\begin{tikzcd}[row sep=small]
				&\mathcal{M}_{u+ \vec{\epsilon}}  \arrow[dd, "f_{u+ \vec{\epsilon}}"] \\
				0=\mathcal{N}_u \arrow[ur, "g_u"] \arrow[dr] & \\
				&\mathcal{N}_{u+ 2\vec{\epsilon}}=0  
			\end{tikzcd}
\end{equation}
\end{center}

Suppose the persistence module $\mathcal{M}$ is $2\epsilon$-trivial. So, by Definition~\ref{epsilontrivial} the transition maps $\varphi_{\mathcal{M}}(u, u + 2\epsilon) = 0$ for all $u \in \mathbb{R}^n$. This implies that all diagrams in (\ref{squarediagrams2}) commute for all $u\preceq v \in \mathbb{R}^n$ and all diagrams in (\ref{trianglediagrams2})  commute for all $u \in \mathbb{R}^n$. Hence, the persistence module $\mathcal{M}$ is $\epsilon$-interleaved with the zero persistence module. Conversely, if the persistence module $\mathcal{M}$ is $\epsilon$-interleaved with the zero persistence module, then all diagrams in (\ref{squarediagrams2}) commute for all $u\preceq v \in \mathbb{R}^n$ and all diagrams in (\ref{trianglediagrams2}) commute for all $u \in \mathbb{R}^n$. Thus, by commutativity of left triangle diagrams in (\ref{trianglediagrams2}) for all $u \in \mathbb{R}^n$, the transition maps $\varphi_{\mathcal{M}}(u, u + 2\epsilon) = 0$ for all $u \in \mathbb{R}^n$. Hence, by Definition~\ref{epsilontrivial} the persistence module $\mathcal{M}$ is $2\epsilon$-trivial.
    
\end{proof}

The following result \cite[Proposition~15]{Dey} will be used later.

\begin{proposition} \label{closureequality}
Let $\mathcal{M}=\mathcal{I}^I$ and $\mathcal{N}=\mathcal{I}^J$ be two interval persistence modules with underlying intervals $I$ and $J$, respectively. Then, $d_I(\mathcal{M}, \mathcal{N} )=d_I(\overline{\mathcal{M}}, \overline{\mathcal{N}})$ where $\overline{\mathcal{M}}\doteq \mathcal{I}^{\Bar{I}}$ and $\overline{\mathcal{N}} \doteq \mathcal{I}^{\Bar{J}}$.
\end{proposition}

From the definition of the interior of an interval, the following corollary is immediate.

\begin{corollary}\label{interiorequality}
Let $\mathcal{M}=\mathcal{I}^I$ and $\mathcal{N}=\mathcal{I}^J$ be two interval persistence modules. Then, $d_I(\mathcal{M}, \mathcal{N} )=d_I(\mathcal{M}^{\mathrm{o}}, \mathcal{N}^{\mathrm{o}})$ where $\mathcal{M}^{\mathrm{o}}\doteq \mathcal{I}^{{I}^{\mathrm{o}}}$ and $\mathcal{N}^{\mathrm{o}}\doteq \mathcal{I}^{{J}^{\mathrm{o}}}$.
\end{corollary}

\subsection{Bottleneck Distance}
In this subsection, we briefly define the bottleneck distance.  Later in Section \ref{bottleneck distance for rectangle}, we extend our results to calculate the bottleneck distance for rectangle decomposable persistence modules.
	
\begin{definition}[\cite{Bjerkevik}]\label{matching}
 Let $A$ and $B$ be multisets of intervals.  An \textbf{$\epsilon$-matching} between $A$ and $B$ is a partial multibijection $\sigma : A \nrightarrow B$ such that
 	  
  \begin{itemize} 
\item [(i)]  for all $I \in A - \coim \sigma$, $\mathcal{I}^{I}$ is $2\epsilon$-trivial,
\item [(ii)] for all $I \in B - \im \sigma$, $\mathcal{I}^{I}$ is $2\epsilon$-trivial,
\item [(iii)] for all $I \in \coim  \sigma$, $\mathcal{I}^{I}$ and $\mathcal{I}^{\sigma(I)}$ are $\epsilon$-interleaved.   
  \end{itemize}	   
\end{definition}
	
If there is an $\epsilon$-matching $\sigma$ between the barcodes $B(\mathcal{M})$ and $B(\mathcal{N})$ of persistence modules $\mathcal{M}$ and $\mathcal{N}$, then we say that $\mathcal{M}$ and $\mathcal{N}$ are \textbf{$\epsilon$-matched}.

Now, we can define the bottleneck distance between the barcodes of multiparameter interval decomposable persistence modules.
	
\begin{definition}[\cite{Bjerkevik}] \label{bot.dis.1}
Let $\mathcal{M}$ and $\mathcal{N}$ be two interval decomposable persistence modules and $S$ be the set of all partial multibijections between the barcodes $B(\mathcal{M})$ and $B(\mathcal{N})$.  Then, the \textbf{bottleneck distance} between interval decomposable persistence modules is defined as          
$$
d_B(\mathcal{M}, \mathcal{N}) \doteq \inf  \{ \epsilon \in [0, +\infty): \mathcal{M} \ \textrm{and} \ \mathcal{N} \ \textrm{are} \ \epsilon \textrm{-matched} \}.
$$ 
If there is no such an $\epsilon$, we put $d_B(\mathcal{M}, \mathcal{N})\doteq +\infty$.
\end{definition} 

Note that, this is also equivalent to setting $\displaystyle d_B(\mathcal{M}, \mathcal{N})=\inf_{\sigma \in S} \cost(\sigma)$.  Also, bear in mind that if $\mathcal{M}$ and $\mathcal{N}$ are finitely presented persistence modules, then one can use minimum in the definition of the bottleneck distance instead of infimum since there always exists a partial multibijection $\overline{\sigma}: B(\mathcal{M})\nrightarrow B(\mathcal{N}) $  such that $ \displaystyle d_B(\mathcal{M}, \mathcal{N})= \cost(\overline{\sigma})$. Following \cite{Cerrii}, we call the partial multibijection $\overline{\sigma}$ \textbf{optimal}.

\section{Computing the Interleaving Distance for Rectangle Persistence Modules}

In this section, we show that the interleaving distance on rectangle persistence modules can be computed by a closed formula using the geometry of underlying rectangles.

In the following, we have results for open rectangle persistence modules. Fortunately, by Proposition~\ref{closureequality} and by Corollary~\ref{interiorequality}, we know that, for any interval persistence modules $\mathcal{M}=\mathcal{I}^I$ and $\mathcal{N}=\mathcal{I}^J$,  $d_I(\mathcal{M}, \mathcal{N} )=d_I(\overline{\mathcal{M}}, \overline{\mathcal{N}})=d_I(\mathcal{M}^{\mathrm{o}}, \mathcal{N}^{\mathrm{o}})$ where $\overline{\mathcal{M}}\doteq \mathcal{I}^{\Bar{I}}$ and $\overline{\mathcal{N}} \doteq \mathcal{I}^{\Bar{J}}$, and $\mathcal{M}^{\mathrm{o}}\doteq \mathcal{I}^{{I}^{\mathrm{o}}}$ and $\mathcal{N}^{\mathrm{o}}\doteq \mathcal{I}^{{J}^{\mathrm{o}}}$. Thus, after proving one for open rectangles, our results are also valid for closed, non-open, and non-closed rectangle persistence modules or, more importantly, finitely presented rectangle persistence modules.

Throughout this section, $\mathcal{M}$ and $\mathcal{N}$ always denote two rectangle persistence modules with underlying rectangles $R_\mathcal{M}=(a_1, b_1) \times (a_2, b_2)$ and $R_\mathcal{N}=(c_1, d_1) \times (c_2, d_2)$ respectively, unless otherwise stated.  Also, we always assume that  that $a_i, c_i\in \mathbb{R}\cup \{-\infty\}$, $b_i,d_i \in \mathbb{R}\cup \{+\infty\}$, and $a_i<b_i$, $c_i<d_i$ for  $i=1, 2$.

\begin{lemma}\label{non-trivialmorphism}
Let $\mathcal{M}$ and $\mathcal{N}$ be two rectangle persistence modules. The set of non-trivial morphisms $f \colon \mathcal{M} \rightarrow \mathcal{N}$ is non-empty if and only if 
$$ c=(c_1, c_2) \preceq a=(a_1, a_2) \prec d=(d_1, d_2)\preceq b=(b_1, b_2).$$ 
Moreover, if such set is non-empty, it contains the morphism $f: \mathcal{M} \rightarrow \mathcal{N}$ such that $f_u$ is the identity map for all $u\in R_\mathcal{M}\cap R_\mathcal{N}$, and the zero map for all  $u\notin R_\mathcal{M}\cap R_\mathcal{N}$.
\end{lemma}

\begin{proof}
Let $ f \colon \mathcal{M} \rightarrow \mathcal{N}$ be a non-trivial morphism. So, there is $u\in \mathbb{R}^2$ such that $f_u\neq 0$. Thus, $a \prec u \prec b$ and $c \prec u \prec d$, and so $a  \prec u   \prec d$. Hence, it is enough to show that $c \preceq  a$ and $d \preceq  b$.

We will first show that $c \preceq  a$. Suppose on the contrary that $c \npreceq  a$, i.e., there exists at least one $i\in\{1,2\}$ for which $c_i > a_i$. Without loss of generality, let $c_1 > a_1$ so that $c_1 > -\infty$. Now, if $u=(u_1, u_2)\notin R_\mathcal{M}$, then $f_u=0$. If $u=(u_1, u_2)\in R_\mathcal{M}$, then $a_2 < u_2$. Thus, there exists $a_2^+ \in \mathbb{R}$ such that $a_2 < a_2^+ \leq u_2$. Now, if $u_1 < c_1$, then $u\notin R_\mathcal{N}$, and so $f_u=0$. Otherwise, if $u_1 \geq c_1$, since we are assuming that $a_1 < c_1$, there exits $a_1^+ \in \mathbb{R}$ such that $a_1 < a_1^+ \leq c_1 \leq u_1$. This implies that $a^+=(a_1^+,a_2^+)\notin  R_\mathcal{N}$ and $a \prec  a^+ \preceq u$. So, $a^+ \in R_\mathcal{M}$. Now, consider the following diagram, where the horizontal maps are the transition morphisms of $\mathcal{M}$ and $\mathcal{N}$, respectively.

\begin{center}
\begin{tikzcd}                              
\mathbb{k}= \mathcal{M}_{a^+} \arrow[r, "\id_{\mathbb{k}}"] \arrow ["0=f_{a^+}", d]    & \mathcal{M}_u=\mathbb{k} \arrow[d, "f_u"] \\
	0=\mathcal{N}_{a^+} \arrow[r]
	& \mathcal{N}_u
\end{tikzcd}
\end{center}

The diagram must be commutative since $f \colon \mathcal{M} \rightarrow \mathcal{N}$ is a morphism. This implies that $f_u=0$ also for the case $u_1 \geq c_1$. Hence, $f_u=0$ for all $u\in \mathbb{R}^2$. So, it contradicts $ f \colon \mathcal{M} \rightarrow \mathcal{N}$ being a non-trivial morphism. Hence, $c \preceq a$.

Still assuming that  $ f \colon \mathcal{M} \rightarrow \mathcal{N}$ is a non-trivial morphism, we will show that $d \preceq  b$. Suppose on the contrary that $d \npreceq  b$, i.e. there exists at least one $i\in\{1,2\}$ for which $d_i > b_i$. Without loss of generality, let $d_1 > b_1$ so that $b_1 < +\infty$. Now, if $u=(u_1, u_2)\notin R_\mathcal{N}$, then $f_u=0$. If $u=(u_1, u_2)\in R_\mathcal{N}$, then $u_2 < d_2$. Thus, there exists $d_2^- \in \mathbb{R}$ such that $u_2 < d_2^- \leq d_2$. Now, if $b_1 < u_1$, then $u\notin R_\mathcal{M}$, and so $f_u=0$. Otherwise, if $b_1 \geq u_1$, since we are assuming that  $b_1 < d_1$, there exits $d_1^- \in \mathbb{R}$ such that $u_1 < b_1 \leq d_1^- < d_1$. This implies that $d^-=(d_1^-,d_2^-)\notin  R_\mathcal{M}$ and $u \preceq d^- \prec d $, and so $d^- \in  R_\mathcal{N}$. Now, consider the following diagram, where the horizontal maps are the interval morphisms of $\mathcal{M}$ and $\mathcal{N}$, respectively.

\begin{center}
			\begin{tikzcd}                              
				\mathcal{M}_u \arrow[r] \arrow ["f_u" , d]   
				& \mathcal{M}_{d^-}=0 \arrow[d, "f_{d^-}=0"] \\
				\mathbb{k}=\mathcal{N}_u \arrow[r, "\id_{\mathbb{k}}"]
				& \mathcal{N}_{d^-}=\mathbb{k} 
			\end{tikzcd}
\end{center}

The diagram must be commutative since $f \colon \mathcal{M} \rightarrow \mathcal{N}$ is a morphism. This implies that $f_u=0$ also for the case $b_1 \geq u_1$. Hence, $f_u=0$ for all $u\in \mathbb{R}^2$. So, it contradicts $ f \colon \mathcal{M} \rightarrow \mathcal{N}$ being a non-trivial morphism. Hence, $d \preceq b$.  \\

For the proof of the converse statement, assuming $c \preceq a \prec d \preceq b$, we can define a non-trivial map $f \colon \mathcal{M} \rightarrow \mathcal{N}$ as follows:
\begin{equation} 
f_u=
\begin{cases} 
     \id_{\mathbb{k}} &  \textrm{if} \ a \preceq u \prec d, \\
      0 & \textrm{otherwise}
\end{cases} \label{nontrivialmorphism}
\end{equation} 
for any $u \in \mathbb{R}^2$. 

Let us now check that it is a valid morphism $\mathcal{M} \rightarrow \mathcal{N}$. For $a\prec u \preceq v \prec d$, we have the following commutative diagram.
\begin{center}
			\begin{tikzcd}                            
				\mathbb{k}=\mathcal{M}_u \arrow["\id_{\mathbb{k}}" ,r] \arrow ["\id_{\mathbb{k}}" , d]   
				& \mathcal{M}_v=\mathbb{k} \arrow["\id_{\mathbb{k}}" ,d] \\
				\mathbb{k}=\mathcal{N}_u \arrow["\id_{\mathbb{k}}", r]
				& \mathcal{N}_v=\mathbb{k} 
			\end{tikzcd} 
\end{center}

For $a \nprec u$ and $u \preceq v$, observe that the following diagram commutes.
\begin{center}
\begin{tikzcd} 
	0=\mathcal{M}_u \arrow["0" ,r] \arrow["0", d]   
	& \mathcal{M}_v \arrow[d] \\
	\mathcal{N}_u \arrow[r]
	&\mathcal{N}_v 
\end{tikzcd} 
\end{center}

For $a \prec u  \preceq v$ and $v \nprec d$,  observe that the following diagram commutes.
\begin{center}
			\begin{tikzcd}                            
				\mathcal{M}_u \arrow[r] \arrow [d]   
				& \mathcal{M}_v \arrow["0", d] \\
				\mathcal{N}_u \arrow["0" , r]
				& \mathcal{N}_v=0 
			\end{tikzcd} 
\end{center}

Therefore, the following diagram commutes for all $u \preceq v \in  \mathbb{R}^2$.
\begin{center}
			\begin{tikzcd}                            
				\mathcal{M}_u \arrow[r] \arrow ["f_u", d]   
				& \mathcal{M}_v \arrow["f_v", d] \\
				\mathcal{N}_u \arrow[r]
				& \mathcal{N}_v 
			\end{tikzcd} 
\end{center}

Hence, $f \colon \mathcal{M} \rightarrow \mathcal{N}$ as in (\ref{nontrivialmorphism}) is the desired non-trivial morphism.
\end{proof}

\begin{corollary} \label{cor-nontrivialmorphism}  
Let $\mathcal{M}$ and $\mathcal{N}$ be two rectangle persistence modules. Let $\vec{\epsilon}=(\epsilon, \epsilon)$ such that $\epsilon \geq 0$ and let $\mathcal{N}(\vec{\epsilon})$ be the $\epsilon$-shift of the persistence module $\mathcal{N}$. There exists a non-trivial morphism $f \colon \mathcal{M} \rightarrow \mathcal{N}(\vec{\epsilon})$, among which the one defined by \[
\displaystyle f_u=
\begin{cases} 
     \id_{\mathbb{k}} &  \textrm{if} \ u \in R_\mathcal{M}\cap R_{\mathcal{N}(\vec{\epsilon})} \\
      0 & \textrm{otherwise} 
\end{cases} \,,
\] if and only if 
$$\displaystyle 
\max \Big\{ \max_{i=1, 2} \{c_i-a_i \}, \max_{i=1, 2} \{ d_i-b_i \} \Big\} \leq \epsilon < \min_{i=1, 2} \{ d_i-a_i \}.
$$

\end{corollary}

\begin{proof}
By Lemma~\ref{non-trivialmorphism}, there is a  non-trivial morphism $f \colon \mathcal{M} \rightarrow \mathcal{N}(\vec{\epsilon})$ if and only if $c- \vec{\epsilon} \preceq a \prec d-\vec{\epsilon} \preceq b$. We note that:

\begin{itemize}
    \item $c- \vec{\epsilon} \preceq a $ if and only if $c_i-\epsilon \leq a_i$ for $i=1, 2$, or equivalently $\epsilon \geq c_i-a_i$ for $i=1, 2$;
     \item $a \prec d-\vec{\epsilon}$ if and only if  $a_i < d_i-\epsilon$ for $i=1, 2$, or equivalently $\epsilon < d_i-a_i$ for $i=1, 2$;
    \item $d- \vec{\epsilon} \preceq b$ if and only if $d_i-\epsilon \leq b_i,$ for $i=1, 2$, or equivalently  $\epsilon \geq d_i-b_i$ for $i=1, 2$.
\end{itemize}

Hence, there is a non-trivial morphism $f \colon \mathcal{M} \rightarrow \mathcal{N}(\vec{\epsilon})$ if and only if 
$$
\displaystyle 
\max \Big\{ \max_{i=1, 2} \{c_i-a_i \}, \max_{i=1, 2} \{ d_i-b_i \} \Big\} \leq \epsilon < \min_{i=1, 2} \{ d_i-a_i \}.
$$
\end{proof}

\begin{lemma} \label{lem:interleaving-trival-mrph}
Let $\mathcal{M}$ and $\mathcal{N}$ be two rectangle persistence modules. We have
\[
d_I(\mathcal{M}, \mathcal{N}) \leq \max \Big\{ \min_{i=1, 2} \frac{b_i-a_i}{2}, \min_{i=1, 2} \frac{d_i-c_i}{2} \Big\}.
\]
\end{lemma}

\begin{proof}

If
\[ \max \Big\{ \min_{i=1, 2} \frac{b_i-a_i}{2}, \min_{i=1, 2} \frac{d_i-c_i}{2} \Big\}=+\infty,
\]
then there is nothing to prove. Thus, let
\begin{equation}\label{lessthaninfinity}
    \epsilon \doteq \max \Big\{ \min_{i=1, 2} \frac{b_i-a_i}{2}, \min_{i=1, 2} \frac{d_i-c_i}{2} \Big\} < +\infty.
\end{equation} 

We note that $\epsilon > 0$. So, we can consider $\mathcal{M}(\vec{\epsilon})$ and $\mathcal{N}(\vec{\epsilon})$. Let us take the maps $f \colon \mathcal{M} \rightarrow \mathcal{N}(\vec{\epsilon})$ and $g \colon \mathcal{N} \rightarrow \mathcal{M}(\vec{\epsilon})$ to be trivial. Thus, all square diagrams~(\ref{squarediagramsfirst})  in the Definition~\ref{interleaving} are commutative. Now consider the following diagram.

\[
\begin{tikzcd}[row sep=4em]
{} & \mathcal{N}_{u+ \vec{\epsilon}} \arrow{dr}{g_{u+ \vec{\epsilon}}=0} \\
\mathcal{M}_u \arrow{ur}{f_u=0} \arrow{rr}{\varphi_\mathcal{M}(u, u+2\vec{\epsilon})} && \mathcal{M}_{u+ 2\vec{\epsilon}}
\end{tikzcd}
\]

Note that, for all points $u\in \mathbb{R}^2$, the diagram above will be commutative if we have $\varphi_\mathcal{M}(u, u+2\vec{\epsilon})=0$. If $a\nprec u$, then it is so because $\mathcal{M}_u=0$. Suppose now that $a\prec u$. By Equation~\ref{lessthaninfinity}, we have 
$$ \epsilon=  \max \Big\{ \min_{i=1, 2} \frac{b_i-a_i}{2}, \min_{i=1, 2} \frac{d_i-c_i}{2} \Big\},$$
thus
$$\epsilon \geq  \min_{i=1, 2} \frac{b_i-a_i}{2}.
$$
Without loss of generality, let 
$$
\displaystyle \min_{i=1, 2} \frac{b_i-a_i}{2}=\frac{b_1-a_1}{2}.
$$ 
So, $a_1+2\epsilon \geq a_1+b_1-a_1=b_1$, which implies $a+2\vec{\epsilon} \nprec b$ so that $\mathcal{M}_{u+2\epsilon}=0$ for every $a\prec u$. Thus, $\varphi_\mathcal{M}(u, u+2\vec{\epsilon})=0$ for all points $u\in \mathbb{R}^2$.

Similarly, one can show that the following diagram is also commutative for all points $u\in \mathbb{R}^2$.

\begin{center}
\begin{tikzcd}[row sep=4em]
{} & \mathcal{M}_{u+ \vec{\epsilon}} \arrow{dr}{f_{u+ \vec{\epsilon}}=0} \\
\mathcal{N}_u \arrow{ur}{g_u=0} \arrow{rr}{\varphi_\mathcal{N}(u, u+2\vec{\epsilon})} && \mathcal{N}_{u+ 2\vec{\epsilon}}
\end{tikzcd}   
\end{center}

Therefore, the trivial morphisms $ f \colon \mathcal{M} \rightarrow \mathcal{N}(\vec{\epsilon})$ and $g \colon \mathcal{N} \rightarrow \mathcal{M}(\vec{\epsilon})$ are $\epsilon$-interleaving morphisms such that all diagrams in Definition~\ref{interleaving} are commutative for every point in $\mathbb{R}^2$. Hence, the rectangle persistence modules $\mathcal{M}$ and $\mathcal{N}$ are $\epsilon$-interleaved, and thus, we can conclude that
\[
d_I(\mathcal{M}, \mathcal{N})\leq \epsilon= \max \Big\{ \min_{i=1, 2} \frac{b_i-a_i}{2}, \min_{i=1, 2} \frac{d_i-c_i}{2} \Big\}.
\]
\end{proof}

\begin{lemma}\label{maxinequality}
Let $\mathcal{M}$ and $\mathcal{N}$ be two rectangle persistence modules. If
\[
\max \Big\{ \| c-a \|_{\infty}, \| d-b \|_{\infty} \Big\} < \max \Big\{ \min_{i=1, 2} \frac{b_i-a_i}{2}, \min_{i=1, 2} \frac{d_i-c_i}{2} \Big\} 
\]
then
\[ 
\displaystyle \max \Big\{ \| c-a \|_{\infty}, \| d-b \|_{\infty} \Big\} < \min \Big\{  \min_{i=1, 2} \{ d_i-a_i \},  \min_{i=1, 2} \{ b_i-c_i \}  \Big\}.
\]
 \end{lemma}  

\begin{proof}
Without loss of generality, let us suppose that  
\[
\min \Big\{  \min_{i=1, 2} \{ d_i-a_i \},  \min_{i=1, 2} \{ b_i-c_i \}  \Big\}=d_1-a_1.
\]
Now, let $\epsilon \doteq \max \big \{ \| c-a \|_{\infty}, \| d-b \|_{\infty} \big \}$, and suppose on the contrary that $ d_1-a_1 \leq \epsilon$. Note that $\frac{d_1-c_1}{2} \leq \frac{d_1-a_1+\epsilon}{2}$ since $ \| c-a \|_{\infty} \leq \epsilon$ by definition of $\epsilon$. Thus, we have $\frac{d_1-c_1}{2} \leq \epsilon$ as $d_1-a_1 \leq \epsilon$ by assumption.
Similarly, note that $\frac{b_1-a_1}{2} \leq \frac{d_1-a_1+\epsilon}{2}$ since  $ \| d-b \|_{\infty}\leq \epsilon$ by definition of $\epsilon$. Thus, we have $\frac{b_1-a_1}{2} \leq  \epsilon$ as $ d_1-a_1 \leq \epsilon$ by assumption. Therefore, we have

\[
\min_{i=1, 2} \frac{b_i-a_i}{2}  \leq \epsilon \ \ \textrm{and} \ \ \min_{i=1, 2} \frac{d_i-c_i}{2} \leq \epsilon \,.
\]

Hence, we have
\[
\max \Big\{ \min_{i=1, 2} \frac{b_i-a_i}{2}, \min_{i=1, 2} \frac{d_i-c_i}{2} \Big\} \leq \epsilon \,,
\]
which contradicts the assumption.
\end{proof}

\begin{lemma} \label{lem:interleaving-nontrival-mrph}
Let $\mathcal{M}$ and $\mathcal{N}$ be two rectangle persistence modules. If
\[
\displaystyle \max \Big\{ \| c-a \|_{\infty}, \| d-b \|_{\infty} \Big\} < \min \Big\{  \min_{i=1, 2} \{ d_i-a_i \},  \min_{i=1, 2} \{ b_i-c_i \}  \Big\},
\]
then
\[
   d_I(\mathcal{M}, \mathcal{N})\leq  \max \Big\{ \| c-a \|_{\infty}, \| d-b \|_{\infty} \Big\}. 
\]

\end{lemma}

\begin{proof}
If
\[ 
\max \Big\{ \| c-a \|_{\infty}, \| d-b \|_{\infty} \Big\}=+\infty,
\]
then there is nothing to prove. Thus, let
\[ 0 \leq \epsilon \doteq \max \Big\{ \| c-a \|_{\infty}, \| d-b \|_{\infty} \Big\} < +\infty.\]
Hence,
\[
\epsilon \ge \max \Big\{ \max_{i=1, 2} \{ c_i-a_i\}, \max_{i=1, 2} \{d_i-b_i \} \Big\}
\]
and, by assumption,
\[
\epsilon < \min \Big\{   \min_{i=1, 2} \{ d_i-a_i \}, \min_{i=1, 2} \{ b_i-c_i \}   \Big\}\le \min_{i=1, 2} \{ d_i-a_i \}. \] 
Analogously,
\[
\epsilon\ge \max \Big\{  \max_{i=1, 2} \{ a_i-c_i\}, \max_{i=1, 2} \{b_i-d_i \} \Big\}
\]
and, by assumption,
\[
\epsilon < \min \Big\{   \min_{i=1, 2} \{ d_i-a_i \}, \min_{i=1, 2} \{ b_i-c_i \}   \Big\} \leq \min_{i=1, 2} \{ b_i-c_i \}. \]

By Corollary~\ref{cor-nontrivialmorphism}, we can take morphisms $f \colon \mathcal{M} \rightarrow \mathcal{N}(\vec{\epsilon})$ and $g \colon \mathcal{N} \rightarrow \mathcal{M}(\vec{\epsilon})$ to be non-trivial and precisely $\id_{\mathbb{k}} :\mathbb{k} \rightarrow \mathbb{k}$ at all points $u\in \mathbb{R}^2$ such that domain and codomain of $f_u$ (respectively $g_u$) are both non-trivial.

Let us show that they make the triangle diagram
\begin{center}
\begin{tikzcd}[row sep=4em]
{} & \mathcal{N}_{u+ \vec{\epsilon}} \arrow{dr}{g_{u+ \vec{\epsilon}}} \\
\mathcal{M}_u \arrow{ur}{f_u} \arrow{rr}{\varphi_\mathcal{M}(u, u+2\vec{\epsilon})} && \mathcal{M}_{u+ 2\vec{\epsilon}}
\end{tikzcd}    
\end{center}
commute for every $u\in \mathbb{R}^2$.

If $\mathcal{M}_u=0$ or $\mathcal{M}_{u+ 2\vec{\epsilon}}=0$, then the triangle diagram is clearly commutative. So, suppose that $\mathcal{M}_u=\mathbb{k}$ and $\mathcal{M}_{u+ 2\vec{\epsilon}}=\mathbb{k}$. Thus, $\varphi_{\mathcal{M}}(u, u+2\vec{\epsilon})=\id_{\mathbb{k}}$ and $a\prec u \prec b-2\vec{\epsilon}$. Thus, we have 
$$
u+\vec{\epsilon} \prec b-\vec{\epsilon} \preceq b- \| d-b \|_{\infty}\cdot (1,1) \preceq d
$$ 
since $  \| d-b \|_{\infty} \leq \epsilon $ by definition of $\epsilon$. Also, we have  
$$
c\prec a + \| c-a \|_{\infty} \cdot (1,1) \preceq a+ \vec{\epsilon} \preceq u+\vec{\epsilon}
$$ 
since $\| c-a \|_{\infty} \leq \epsilon$ by definition of $\epsilon$. These two inequalities imply that $c \prec u+\vec{\epsilon} \prec d$. It follows that $\mathcal{N}_{u+ \vec{\epsilon}}=\mathbb{k}$, and hence $f_u=\id_{\mathbb{k}}$ and $g_{u+ \vec{\epsilon}}=\id_{\mathbb{k}}$ make the diagram commute since we are assuming that $\mathcal{M}_u=\mathbb{k}$ and $\mathcal{M}_{u+ 2\vec{\epsilon}}=\mathbb{k}$. Hence, the triangle diagram above commutes for all points $u \in \mathbb{R}^2$. Analogously, the triangle diagram below is commutative for all points $u \in \mathbb{R}^2$.

\[
\begin{tikzcd}[row sep=4em]
{} & \mathcal{M}_{u+ \vec{\epsilon}} \arrow{dr}{f_{u+ \vec{\epsilon}}} \\
\mathcal{N}_u \arrow{ur}{g_u} \arrow{rr}{\varphi_\mathcal{N}(u, u+2\vec{\epsilon})} && \mathcal{N}_{u+ 2\vec{\epsilon}}
\end{tikzcd}
\]

Therefore, $f,g$ form an $\epsilon$-interleaving pair, implying that 
\[
d_I(\mathcal{M}, \mathcal{N}) \leq \epsilon= \max \Big\{ \| c-a \|_{\infty}, \| d-b \|_{\infty} \Big\}.
\]

\end{proof}

We are ready to give an upper bound for the interleaving distance between rectangle persistence modules.

\begin{theorem} \label{computationofinterleavingdistance}
Let $\mathcal{M}$ and $\mathcal{N}$ be two rectangle persistence modules. We have
\small \[
d_I(\mathcal{M}, \mathcal{N})\leq \min \Bigg\{ \max \Big\{ \min_{i=1, 2} \frac{b_i-a_i}{2}, \min_{i=1, 2} \frac{d_i-c_i}{2} \Big\}, \max \Big\{ \| c-a \|_{\infty}, \| d-b \|_{\infty} \Big\}  \Bigg\}.
\]
\end{theorem}

\begin{proof}
First, suppose that
\[
\max \Big\{ \min_{i=1, 2} \frac{b_i-a_i}{2}, \min_{i=1, 2} \frac{d_i-c_i}{2} \Big\} \leq \max \Big\{ \| c-a \|_{\infty}, \| d-b \|_{\infty} \Big\}.
\]
By Lemma~\ref{lem:interleaving-trival-mrph}, we know that 
\[
d_I(\mathcal{M}, \mathcal{N})\leq \max \Big\{ \min_{i=1, 2} \frac{b_i-a_i}{2}, \min_{i=1, 2} \frac{d_i-c_i}{2} \Big\}.
\]
Hence, the result follows in this case.

Now, suppose that 
\[
\max \Big\{ \| c-a \|_{\infty}, \| d-b \|_{\infty} \Big\}  < \max \Big\{ \min_{i=1, 2} \frac{b_i-a_i}{2}, \min_{i=1, 2} \frac{d_i-c_i}{2} \Big\}.
\]
Then, by Lemma~\ref{maxinequality}, we know that
\[ 
\displaystyle \max \Big\{ \| c-a \|_{\infty}, \| d-b \|_{\infty} \Big\} < \min \Big\{  \min_{i=1, 2} \{ d_i-a_i \},  \min_{i=1, 2} \{ b_i-c_i \}  \Big\}.
\]
Therefore, by Lemma~\ref{lem:interleaving-nontrival-mrph},
\[
d_I(\mathcal{M}, \mathcal{N}) \leq \max \Big\{ \| c-a \|_{\infty}, \| d-b \|_{\infty} \Big\}.
\]

Hence, in any case
\small \[
d_I(\mathcal{M}, \mathcal{N}) \leq \min \Bigg\{ \max \Big\{ \min_{i=1, 2} \frac{b_i-a_i}{2}, \min_{i=1, 2} \frac{d_i-c_i}{2} \Big\}, \max \Big\{ \| c-a \|_{\infty}, \| d-b \|_{\infty} \Big\}  \Bigg\}.
\]

\end{proof}

Next, we will prove the converse of the inequality of Theorem~\ref{computationofinterleavingdistance}.  In particular, if we have two rectangle persistence modules, $\mathcal{M}$ and $\mathcal{N}$, with underlying rectangles $R_\mathcal{M}=(a_1, b_1) \times (a_2, b_2)$ and $R_\mathcal{N}=(c_1, d_1) \times (c_2, d_2)$, respectively, then we want to show that
\small \[
d_I(\mathcal{M}, \mathcal{N}) \geq  \min \Bigg\{ \max \Big\{ \min_{i=1, 2} \frac{b_i-a_i}{2}, \min_{i=1, 2} \frac{d_i-c_i}{2} \Big\}, \max \Big\{ \| c-a \|_{\infty}, \| d-b \|_{\infty} \Big\}  \Bigg\}.
\]

\begin{lemma}\label{forgtrivialmorphism}
Let $\mathcal{M}$ and $\mathcal{N}$ be $\epsilon$-interleaved persistence modules with two interleaving morphisms $f \colon \mathcal{M} \rightarrow \mathcal{N}(\vec{\epsilon})$ and $g \colon \mathcal{N} \rightarrow \mathcal{M}(\vec{\epsilon})$. If $f$ or $g$ is a trivial morphism, then
\[
\epsilon \geq \max \Big\{ \min_{i=1, 2} \frac{b_i-a_i}{2}, \min_{i=1, 2} \frac{d_i-c_i}{2} \Big\}. 
\]
\end{lemma}

\begin{proof}
Without loss of generality, suppose that $f \colon \mathcal{M} \rightarrow \mathcal{N}(\vec{\epsilon})$ is a trivial morphism. Thus, $f_u \colon \mathcal{M}_u \rightarrow \mathcal{N}_{u+\vec{\epsilon}}$ is a zero map for every $u\in \mathbb{R}^2$. By assumption, we know that $\mathcal{M}$ and $\mathcal{N}$ are $\epsilon$-interleaved. Hence, both of the diagrams below must be commutative for every $u\in \mathbb{R}^2$.
\begin{center}
\begin{tikzcd}[row sep=4em]
{} & \mathcal{N}_{u+ \vec{\epsilon}} \arrow{dr}{g_{u+ \vec{\epsilon}}} \\
\mathcal{M}_u \arrow{ur}{0=f_u} \arrow{rr}{\varphi_\mathcal{M}(u, u+2\vec{\epsilon})} && \mathcal{M}_{u+ 2\vec{\epsilon}}
\end{tikzcd}
\begin{tikzcd}[row sep=4em]
{} & \mathcal{M}_{u+ \vec{\epsilon}} \arrow{dr}{f_{u+ \vec{\epsilon}}=0} \\
\mathcal{N}_u \arrow{ur}{g_u} \arrow{rr}{\varphi_\mathcal{N}(u, u+2\vec{\epsilon})} && \mathcal{N}_{u+ 2\vec{\epsilon}}
\end{tikzcd}   
\end{center}
As a result, both transition maps ${\varphi_\mathcal{M}(u, u+2\vec{\epsilon})}$ and ${\varphi_\mathcal{N}(u, u+2\vec{\epsilon})}$ are zero linear maps for every $u\in \mathbb{R}^2$. Therefore, both $\mathcal{M}$ and $\mathcal{N}$ are $2\epsilon$-trivial persistence modules. This implies that $\displaystyle 2\epsilon \geq \min_{i=1, 2} \{b_i-a_i\}$ and $\displaystyle 2\epsilon \geq\min_{i=1, 2} \{d_i-c_i\}$. Hence, we can conclude that
\[
\epsilon \geq \max \Big\{ \min_{i=1, 2} \frac{b_i-a_i}{2}, \min_{i=1, 2} \frac{d_i-c_i}{2} \Big\}. 
\]

\end{proof}

\begin{lemma}\label{fandgaretrivialmorphisms}
Let $\mathcal{M}$ and $\mathcal{N}$ be $\epsilon$-interleaved persistence modules with two interleaving morphisms $f \colon \mathcal{M} \rightarrow \mathcal{N}(\vec{\epsilon})$ and $g \colon \mathcal{N} \rightarrow \mathcal{M}(\vec{\epsilon})$. If $f$ and $g$ are non-trivial morphisms, then
\[
\epsilon \geq \max \Big\{ \| c-a \|_{\infty}, \| d-b \|_{\infty} \Big\}. 
\]   
\end{lemma}

\begin{proof}
From Corollary~\ref{cor-nontrivialmorphism}, since $f \colon \mathcal{M} \rightarrow \mathcal{N}(\vec{\epsilon})$ is a non-trivial morphism, we get
\[\displaystyle \epsilon \geq \max \Big\{ \max_{i=1, 2} \{c_i-a_i \}, \max_{i=1, 2} \{ d_i-b_i \} \Big\}.
\]
Analogously, since $g \colon \mathcal{N} \rightarrow \mathcal{M}(\vec{\epsilon})$ is a non-trivial morphism, we have
\[\displaystyle \epsilon \geq \max \Big\{ \max_{i=1, 2} \{a_i-c_i \}, \max_{i=1, 2} \{ b_i-d_i \} \Big\}.
\]
Hence, we can conclude that
\[
\epsilon \geq \max \Big\{ \| c-a \|_{\infty}, \| d-b \|_{\infty} \Big\} . 
\]
\end{proof}

We are ready to provide our first main result, which is as follows:

\begin{thma} \label{interleavingdistanceisequaltostarvalue}
Let $\mathcal{M}$ and $\mathcal{N}$ be two rectangle persistence modules. We have
\small \[
d_I(\mathcal{M}, \mathcal{N}) = \min \Bigg\{ \max \Big\{ \min_{i=1, 2} \frac{b_i-a_i}{2}, \min_{i=1, 2} \frac{d_i-c_i}{2} \Big\}, \max \Big\{ \| c-a \|_{\infty}, \| d-b \|_{\infty} \Big\}    \Bigg\}.
\]   
\end{thma}

\begin{proof}

First, note that by Theorem~\ref{computationofinterleavingdistance}, we already have 
\small \[
d_I(\mathcal{M}, \mathcal{N}) \leq \min \Bigg\{ \max \Big\{ \min_{i=1, 2} \frac{b_i-a_i}{2}, \min_{i=1, 2} \frac{d_i-c_i}{2} \Big\}, \max \Big\{ \| c-a \|_{\infty}, \| d-b \|_{\infty} \Big\}    \Bigg\}.
\]   
For the remaining part of the proof, let  $\mathcal{M}$ and $\mathcal{N}$ be two $\epsilon$-interleaved persistence modules with interleaving morphisms $f \colon \mathcal{M} \rightarrow \mathcal{N}(\vec{\epsilon})$ and $g \colon \mathcal{N} \rightarrow \mathcal{M}(\vec{\epsilon})$. Suppose first that $f$ or $g$ is a trivial morphism. By Lemma~\ref{forgtrivialmorphism}, we know that
\[
\epsilon \geq \max \Big\{ \min_{i=1, 2} \frac{b_i-a_i}{2}, \min_{i=1, 2} \frac{d_i-c_i}{2} \Big\}. 
\] 
Thus, we have
\[
\epsilon \geq \min \Bigg\{ \max \Big\{ \min_{i=1, 2} \frac{b_i-a_i}{2}, \min_{i=1, 2} \frac{d_i-c_i}{2} \Big\}, \max \Big\{ \| c-a \|_{\infty}, \| d-b \|_{\infty} \Big\}    \Bigg\}. 
\]
Now, suppose that $f$ and $g$ are non-trivial morphisms. By Lemma~\ref{fandgaretrivialmorphisms}, we know that 
\[
\epsilon \geq \max \Big\{ \| c-a \|_{\infty}, \| d-b \|_{\infty} \Big\} .
\]
Thus, again, we have
\[
\epsilon \geq \min \Bigg\{ \max \Big\{ \min_{i=1, 2} \frac{b_i-a_i}{2}, \min_{i=1, 2} \frac{d_i-c_i}{2} \Big\}, \max \Big\{ \| c-a \|_{\infty}, \| d-b \|_{\infty} \Big\}    \Bigg\}. 
\]
Therefore in both cases, we have
\[
\epsilon \geq \min \Bigg\{ \max \Big\{ \min_{i=1, 2} \frac{b_i-a_i}{2}, \min_{i=1, 2} \frac{d_i-c_i}{2} \Big\}, \max \Big\{ \| c-a \|_{\infty}, \| d-b \|_{\infty} \Big\}    \Bigg\}.
\]
Hence, we can immediately conclude that 
\small \[
d_I(\mathcal{M}, \mathcal{N}) \geq \min \Bigg\{ \max \Big\{ \min_{i=1, 2} \frac{b_i-a_i}{2}, \min_{i=1, 2} \frac{d_i-c_i}{2} \Big\}, \max \Big\{ \| c-a \|_{\infty}, \| d-b \|_{\infty} \Big\}    \Bigg\},
\]
which finishes the proof.
\end{proof}

\begin{remark}
Thanks to  Proposition~\ref{closureequality} and Corollary~\ref{interiorequality}, Theorem A holds independently of whether the underlying rectangles are open, closed or neither.
\end{remark}

We can generalize Theorem A to any $n$-paramater rectangle persistence modules by imitating the previous results.

\begin{corollary}\label{interleavingfor$n$-parameter}
    Let $\mathcal{M}$ and $\mathcal{N}$ be $n$-parameter rectangle persistence modules with underlying rectangles $R_\mathcal{M}=(a_1, b_1) \times \ldots \times (a_n, b_n)$ and $R_\mathcal{N}=(c_1, d_1) \times  \ldots \times (c_n, d_n)$ where $a=(a_1, \ldots, a_n)$, $b=(b_1, \ldots, b_n)$, $c=(c_1, \ldots, c_n)$ and $ d=(d_1, \ldots, d_n)$, respectively. Then,
\small \[
d_I(\mathcal{M}, \mathcal{N}) = \min \Bigg\{ \max \Big\{ \min_{1 \leq i \leq n} \frac{b_i-a_i}{2}, \min_{1 \leq i \leq n} \frac{d_i-c_i}{2} \Big\}, \max \Big\{ \| c-a \|_{\infty}, \| d-b \|_{\infty} \Big\}  \Bigg\}.
\]  
\end{corollary}

\section{Computing the Bottleneck Distance for Rectangle Decomposable Persistence Modules} \label{bottleneck distance for rectangle}
In this section, we first give an equivalent expression of Definition~\ref{matching} for the calculation of the bottleneck distance for interval decomposable persistence modules.  Then, we present our result for the calculation of the bottleneck distance for rectangle decomposable persistence modules.

\begin{proposition} \label{bot.dis.2}
Let $\mathcal{M}$ and $\mathcal{N}$ be two finitely presented interval decomposable persistence modules with given decompositions
$$
\displaystyle \mathcal{M}=\bigoplus_{I \in B(\mathcal{M})} \mathcal{I}^{I}   \ \textrm{and} \ \displaystyle \mathcal{N}=\bigoplus_{J \in B(\mathcal{N})} \mathcal{I}^{J}
$$ 
where summands $\mathcal{I}^{I}$ and $\mathcal{I}^{J}$ are interval persistence modules with underlying intervals $I$ and $J$, respectively. Let $S$ be the set of all partial multibijections between the barcodes $B(\mathcal{M})$ and $B(\mathcal{N})$. Let $\sigma \in S$ and let $\mathbf{I}'(\sigma)=B(\mathcal{M})- \coim \sigma$, $\mathbf{J}'(\sigma)=B(\mathcal{N})- \im \sigma$.  Then, the bottleneck distance $d_B(\mathcal{M}, \mathcal{N})$ is equal to	

\begin{equation} \label{computationofbottleneckequation}
\min_{\sigma \in S}  \max \Big\{  \max_{I \in \coim \sigma} \big\{ d_I(\mathcal{I}^{I}, \mathcal{I}^{\sigma(I)}) \big\}, \max_{I \in \mathbf{I}'(\sigma)} \big\{ d_I(\mathcal{I}^{I}, 0) \big\}, \max_{J \in \mathbf{J}'(\sigma)} \big\{ d_I(0, \mathcal{I}^{J}) \big\} \Big\}.  
\end{equation}   

\end{proposition}

\begin{proof}
Since $\mathcal{M}$ and $\mathcal{N}$ are finitely presented interval decomposable persistence modules, $S$ is finite as well as $B(\mathcal{M})$ and $B(\mathcal{N})$. If $d_B(\mathcal{M}, \mathcal{N})=+\infty$, then the quantity in (\ref{computationofbottleneckequation}) is less than or equal to $d_B(\mathcal{M}, \mathcal{N})$. If $d_B(\mathcal{M}, \mathcal{N})< +\infty$, consider a partial multibijection $\Bar{\sigma} \colon B(\mathcal{M}) \nrightarrow B(\mathcal{N}) $ between the barcodes $B(\mathcal{M})$ and $B(\mathcal{N})$ such that $\Bar{\sigma}$ is an $\epsilon$-matching. Therefore,

\begin{itemize} 
\item [(i)]  for all $I \in B(\mathcal{M}) - \coim \Bar{\sigma}$, $\mathcal{I}^{I}$ is $2\epsilon$-trivial,
\item [(ii)] for all $J \in B(\mathcal{N}) - \im \Bar{\sigma}$, $\mathcal{I}^{J}$ is $2\epsilon$-trivial,
\item [(iii)] for all $I \in \coim  \Bar{\sigma}$, $\mathcal{I}^{I}$ and $\mathcal{I}^{\Bar{\sigma}(I)}$ are $\epsilon$-interleaved.   
\end{itemize}

Now, by Proposition~\ref{2epsilontrivialiffepsiloninterleaved}, the interval persistence module $\mathcal{I}^{I}$ is $2\epsilon$-trivial if and only if $\mathcal{I}^{I}$ is $\epsilon$-interleaved with the zero persistence module for all $I \in B(\mathcal{M}) - \coim \Bar{\sigma}$. Thus, $\epsilon \geq d_I(\mathcal{I}^{I}, 0)$ for all $I \in B(\mathcal{M}) - \coim \Bar{\sigma}$.

Similarly, by Proposition~\ref{2epsilontrivialiffepsiloninterleaved}, the interval persistence module $\mathcal{I}^{J}$ is $2\epsilon$-trivial if and only if $\mathcal{I}^{J}$ is $\epsilon$-interleaved with the zero persistence module for all $J \in B(\mathcal{N})- \im \Bar{\sigma}$. Thus, $\epsilon \geq d_I(0, \mathcal{I}^{J})$ for all $J \in B(\mathcal{N}) - \im \Bar{\sigma}$.

Also,  $\mathcal{I}^{I}$ and $\mathcal{I}^{\Bar{\sigma}(I)}$ are $\epsilon$-interleaved for all $I \in \coim  \Bar{\sigma}$. Thus, $\epsilon \geq d_I(\mathcal{I}^{I}, \mathcal{I}^{\Bar{\sigma}(I)})$ for all $I \in \coim  \Bar{\sigma}$.

Therefore, 
\[
\displaystyle \epsilon \geq \max \big \{  \max_{I \in \coim \Bar{\sigma}} \{d_I(\mathcal{I}^{I}, \mathcal{I}^{\Bar{\sigma}(I)})\}, \max_{I \in \mathbf{I}'(\sigma)} \{ d_I(\mathcal{I}^{I}, 0)\}, \max_{J \in \mathbf{J}'(\sigma)} \{d_I(0, \mathcal{I}^{J})\} \big \}.
\]
It follows that 
\[
\displaystyle \epsilon \geq \min_{\sigma \in S} \big ( \max \big \{  \max_{I \in \coim \sigma} \{d_I(\mathcal{I}^{I}, \mathcal{I}^{\sigma(I)})\}, \max_{I \in \mathbf{I}'(\sigma)} \{ d_I(\mathcal{I}^{I}, 0)\}, \max_{J \in \mathbf{J}'(\sigma)} \{d_I(0, \mathcal{I}^{J})\} \big \}  \big ).
\]

By the arbitrariness of $\epsilon$, we deduce that
\[
\displaystyle d_B(\mathcal{M}, \mathcal{N}) \geq \min_{\sigma \in S}\max \Big\{  \max_{I \in \coim \sigma} \big\{ d_I(\mathcal{I}^{I}, \mathcal{I}^{\sigma(I)}) \big\}, \max_{I \in \mathbf{I}'(\sigma)} \big\{ d_I(\mathcal{I}^{I}, 0) \big\}, \max_{J \in \mathbf{J}'(\sigma)} \big\{ d_I(0, \mathcal{I}^{J}) \big\} \Big\}.
\]

Conversely, let us set
$$\displaystyle \epsilon=\min_{\sigma \in S} \max \Big\{  \max_{I \in \coim \sigma} \big\{ d_I(\mathcal{I}^{I}, \mathcal{I}^{\sigma(I)}) \big\}, \max_{I \in \mathbf{I}'(\sigma)} \big\{ d_I(\mathcal{I}^{I}, 0) \big\}, \max_{J \in \mathbf{J}'(\sigma)} \big\{ d_I(0, \mathcal{I}^{J}) \big\} \Big\}.$$ 
Let $\Bar{\sigma}$ be a partial multibijection such that 
$$\displaystyle \epsilon=\max \big \{  \max_{I \in \coim \Bar{\sigma}} \{d_I(\mathcal{I}^{I}, \mathcal{I}^{\Bar{\sigma}(I)})\}, \max_{I \in \mathbf{I}'(\sigma)} \{ d_I(\mathcal{I}^{I}, 0)\}, \max_{J \in \mathbf{J}'(\sigma)} \{d_I(0, \mathcal{I}^{J})\} \big \}.$$ It follows that $\displaystyle \epsilon \geq d_I(\mathcal{I}^{I}, \mathcal{I}^{\Bar{\sigma}(I)})$ for all $I \in \coim  \Bar{\sigma}$. Similarly, $\epsilon \geq d_I(\mathcal{I}^{I}, 0)$ for all $I \in  \mathbf{I}'(\sigma)= B(\mathcal{M}) - \coim \Bar{\sigma}$ and $\epsilon \geq d_I(0, \mathcal{I}^{J})$ for all $J \in \mathbf{J}'(\sigma)= B(\mathcal{N}) - \im \Bar{\sigma}$. Clearly, $\mathcal{I}^{I}$ is $\epsilon$-interleaved with the zero persistence module for each $I \in  B(\mathcal{M}) - \coim \Bar{\sigma}$ and $\mathcal{I}^{J}$ is $\epsilon$-interleaved with the zero persistence module for each $J \in  B(\mathcal{N}) - \im \Bar{\sigma}$. So, by Proposition~\ref{2epsilontrivialiffepsiloninterleaved}, $\mathcal{I}^{I}$ is $2\epsilon$-trivial for each $I \in  B(\mathcal{M}) - \coim \Bar{\sigma}$ and $\mathcal{I}^{J}$ is $2\epsilon$-trivial for each $J \in  B(\mathcal{N}) - \im \Bar{\sigma}$. So we have
\begin{itemize} 
\item [(i)]  for all $I \in B(\mathcal{M}) - \coim \Bar{\sigma}$, $\mathcal{I}^{I}$ is $2\epsilon$-trivial,
\item [(ii)] for all $J \in B(\mathcal{N}) - \im \Bar{\sigma}$, $\mathcal{I}^{J}$ is $2\epsilon$-trivial,
\item [(iii)] for all $I \in \coim  \Bar{\sigma}$, $\mathcal{I}^{I}$ and $\mathcal{I}^{\Bar{\sigma}(I)}$ are $\epsilon$-interleaved.   
\end{itemize}

Thus, $\Bar{\sigma}$ is an $\epsilon$-matching. So, by Definition~\ref{bot.dis.1}, $\displaystyle \epsilon \geq d_B(\mathcal{M}, \mathcal{N})$.
Hence, the claim is proved.
\end{proof}

The following result shows that the bottleneck distance is equal to the interleaving distance if $\mathcal{M}$ and $\mathcal{N}$ are two $n$-parameter interval persistence modules (for more details of the proof, we refer to the reader to \cite[Proposition 4.3.1]{Batan}).

\begin{corollary} \label{indecomposableequality}
If $\mathcal{M}=\mathcal{I}^I$ and $\mathcal{N}=\mathcal{I}^J$ are two $n$-parameter interval persistence modules, then $d_B(\mathcal{M}, \mathcal{N}) = d_I(\mathcal{M}, \mathcal{N})$.
\end{corollary}

\begin{proof}
It is well known that $d_I(\mathcal{M}, \mathcal{N})\leq d_B(\mathcal{M}, \mathcal{N})$. Thus, it is enough to show that $d_B(\mathcal{M}, \mathcal{N}) \leq d_I(\mathcal{M}, \mathcal{N})$. Suppose that the interval persistence modules $\mathcal{M}=\mathcal{I}^I$ and $\mathcal{N}=\mathcal{I}^J$ are $\epsilon$-interleaved. Now, consider the bijection $\sigma \colon \{I\} \rightarrow \{J\} $ with $\im \sigma = \{J\}$ and $\coim \sigma = \{I\}$. Therefore, $B(\mathcal{M}) - \coim \sigma = \emptyset$ and $B(\mathcal{N}) - \im \sigma = \emptyset$. Now, by Definition~\ref{matching}, the bijection $\sigma \colon \{I\} \rightarrow \{J\} $ is an $\epsilon$-matching since $\mathcal{I}^I$ and $\mathcal{I}^J$ are $\epsilon$-interleaved. In particular, we have $d_B(\mathcal{M}, \mathcal{N})\leq d_I(\mathcal{M}, \mathcal{N})$.

\end{proof}

Now, we are ready to give an exact computation of the bottleneck distance between $n$-parameter rectangle decomposable persistence modules. For ease of notation, we give the exact computation for $n=2$. 

Let $\mathcal{M}$ and $\mathcal{N}$ be two rectangle decomposable persistence modules with decompositions 
$$
\displaystyle \mathcal{M}=\bigoplus_{i=1}^{m}\mathcal{M}_i ~\textrm{and}~ \displaystyle \mathcal{N}=\bigoplus_{j=1}^{k}\mathcal{N}_j.$$  
Let us denote the underlying rectangle of the persistence module $\mathcal{M}_i$ by $R_i=(a_1^i, b_1^i)\times(a_2^i, b_2^i)$
and $\mathcal{N}_j$ by $Q_j=(c_1^j, d_1^j)\times(c_2^j, d_2^j)$
where $a^i=(a_1^i, a_2^i)$, $b^i=(b_1^i, b_2^i)$, $c^j=(c_1^j, c_2^j)$ and $d^j=(d_1^j, d_2^j)$.

\begin{thmb}\label{bottleneck for rectangle}
Let $\mathcal{M}$ and $\mathcal{N}$ be two rectangle decomposable persistence modules given as above. Let $S$ be the set of all partial multibijections between the barcodes $B(\mathcal{M})$ and $B(\mathcal{N})$. Let $\sigma \in S$ and let $\mathbf{I}'(\sigma)=B(\mathcal{M})- \coim \sigma$, $\mathbf{J}'(\sigma)=B(\mathcal{N})- \im \sigma$. Then, the bottleneck distance is equal to {\footnotesize
\[
 \displaystyle \min_{\sigma \in S} \max \left\{  \max_{R_i \in \coim \sigma} \Big\{  d_I\Big(\mathcal{M}_i, \sigma(\mathcal{M}_i)\Big) \Big\}, \max_{R_{i'} \in \mathbf{I}'(\sigma)} \Big\{ \min_{s=1,2} \frac{b_s^{i'}-a_s^{i'}}{2} \Big\}, \max_{Q_{j'} \in \mathbf{J}'(\sigma)} \Big\{ \min_{s=1,2} \frac{d_s^{j'}-c_s^{j'}}{2} \Big\} \right\}
\]}
where the interleaving distance between $\mathcal{M}_i$ and $\sigma(\mathcal{M}_i)$ is equal to {\footnotesize
\[
\min \Bigg\{ \max \Big\{ \min_{s=1,2} \frac{b_s^i-a_s^i}{2}, \min_{s=1,2} \frac{d_s^{\sigma(i)}-c_s^{\sigma(i)}}{2} \Big\}, \max \Big\{ \| c^{\sigma(i)}-a^i \|_{\infty}, \| d^{\sigma(i)}-b^i \|_{\infty} \Big\}  \Bigg\}.
\]}
\end{thmb}

\begin{proof}
By Proposition~\ref{bot.dis.2}, the bottleneck distance between two rectangle decomposable persistence modules $\mathcal{M}$ and $\mathcal{N}$ with decompositions $$
\displaystyle \mathcal{M}=\bigoplus_{i=1}^{m}\mathcal{M}_i ~\textrm{and}~ \displaystyle \mathcal{N}=\bigoplus_{j=1}^{k}\mathcal{N}_j$$  can be computed as
{\footnotesize
\[
\displaystyle \min_{\sigma \in S} \max \left\{  \max_{R_i \in \coim \sigma} \Big\{  d_I\Big(\mathcal{M}_i, \sigma(\mathcal{M}_i)\Big) \Big\}, \max_{R_{i'} \in \mathbf{I}'(\sigma)} \Big\{ d_I\Big(\mathcal{M}_{i'}, 0\Big) \Big\}, \max_{Q_{j'} \in \mathbf{J}'(\sigma)} \Big\{  d_I\Big(0, \mathcal{N}_{j'}\Big)   \Big\} \right\}.
\]}

By Propositions~\ref{epsilontrivialrectangle} and~\ref{2epsilontrivialiffepsiloninterleaved}, it follows that
\[
d_I\Big(\mathcal{M}_{i'}, 0\Big)=\min_{s=1,2} \frac{b_s^{i'}-a_s^{i'}}{2} \ \ \textrm{and} \ \ d_I\Big(0,  \mathcal{N}_{j'}\Big)=\min_{s=1,2} \frac{d_s^{j'}-c_s^{j'}}{2}.
\]

Finally, by Theorem A, we know that the interleaving distance between $\mathcal{M}_i$ and $\sigma(\mathcal{M}_i)$ is equal to
{\footnotesize \[ 
\displaystyle \min \Bigg\{ \max \Big\{ \min_{s=1,2} \frac{b_s^i-a_s^i}{2}, \min_{s=1,2} \frac{d_s^{\sigma(i)}-c_s^{\sigma(i)}}{2} \Big\}, \max \Big\{ \| c^{\sigma(i)}-a^i \|_{\infty}, \| d^{\sigma(i)}-b^i \|_{\infty} \Big\}  \Bigg\}.
\]}
\end{proof}

The next stage of our ongoing research is to get an analogous formula for the case of interval modules.  This is not a straightforward task:  An underlying rectangle can be represented by just two corner points, but this is no longer true and is more complicated for an underlying interval of an interval persistence module.  Moreover, two rectangles or their shiftings intersect along a single component, which is again a rectangle, but for intervals or their shiftings, the number of intersection components can be more than one.

\bibliographystyle{amsplain}

\begin{thebibliography}{10}

         \bibitem{aktas-akbas-fatmaoui}
     Aktas, M. E., Akbas, E. and Fatmaoui, A. E. (2019). Persistence homology of networks: methods and applications. Applied Network Science, 4(1), 1-28. 

        \bibitem{Atiyah}
     Atiyah, M. F. (1956). On the Krull-Schmidt theorem with application to sheaves. Bulletin de la Société mathématique de France, 84, 307-317.

         \bibitem{Batan}
     Batan, M. A. (2024). Distances for Multiparameter Persistence Modules.  PhD Thesis, Middle East Technical University. 
         

         \bibitem{Bjerkevik}
    Bjerkevik, H. B. (2021). On the stability of interval decomposable persistence modules.  Discrete \& Computational Geometry, 66,  92--121.

         \bibitem{bjerkevik-botnan}
    Bjerkevik, H. B. and Botnan, M. B. (2018). Computational Complexity of the Interleaving Distance. In 34th International Symposium on Computational Geometry, 99, 13:1--13:15.
    
         \bibitem{botnan-lebovici-oudot}
   Botnan, M. B., Lebovici, V. and Oudot, S. (2022). On Rectangle-Decomposable $2$-Parameter Persistence Modules.  Discrete \& Computational Geometry, 68, 1078--1101 (2022). 

        \bibitem{Botnan}
    Botnan, M. and Lesnick, M. (2018). Algebraic stability of zigzag persistence modules. Algebraic and geometric topology, 18(6), 3133--3204.

        \bibitem{carlsson-zomorodian}
    Carlsson, G. and  Zomorodian, A. (2009) The theory of multidimensional persistence. Discrete \& Computational Geometry, 42(1), 71--93.

       \bibitem{chazal et al}
    Chazal, F.,   Cohen-Steiner, D.,  Glisse, M.,   Guibas, L. J. and  Oudot, S. Y. (2009) Proximity of persistence modules and their diagrams. In Proceedings of the 25th annual Symposium on Computational Geometry, 237--246. 


         \bibitem{Cerrii}
	Cerri, A., Ethier, M. and Frosini, P. (2016). The coherent matching distance in 2D persistent homology. In International Workshop on Computational Topology in Image Context, 216--227.

         \bibitem{Corbet}
	Corbet, R. and Kerber, M. (2018). The representation theorem of persistence revisited and generalized. Journal of Applied and Computational Topology, 2(1-2), 1-31.
        

        \bibitem{Dey}
     Dey, T. K. and Xin, C. (2018). Computing bottleneck distance for $2$-d interval decomposable modules. In 34th International Symposium on Computational Geometry. 
 
        \bibitem{ismail-noorani-ismail-razak-alias}
     Ismail, M. S., Noorani, M. S. M., Ismail, M., Razak, F. A. and Alias, M. A. (2022). Early warning signals of financial crises using persistent homology. Physica A: Statistical Mechanics and Its Applications, 586, 126459.


        \bibitem{Lesnick}
     Lesnick, M. (2015). The theory of the interleaving distance on multidimensional persistence modules. Foundations of Computational Mathematics, 15(3), 613--650.


        \bibitem{Oudot}
     Oudot, S. Y. (2017). Persistence theory: from quiver representations to data analysis (Vol. 209). American Mathematical Society.


       \bibitem{topaz-ziegelmeier-halverson}
     Topaz, C. M., Ziegelmeier, L. and Halverson, T. (2015). Topological data analysis of biological aggregation models. PloS one, 10(5), e0126383.


    


\end{thebibliography}
	\providecommand{\bysame}{\leavevmode\hbox
		to3em{\hrulefill}\thinspace}
	\providecommand{\MR}{\relax\ifhmode\unskip\space\fi MR }
	\providecommand{\MRhref}[2]{
		\href{http://www.ams.org/mathscinet-getitem?mr=#1}{#2}
	} \providecommand{\href}[2]{#2}

\end{document}